\documentclass[a4paper,11pt]{amsart}
\usepackage{amssymb, amsthm, enumerate, graphicx, amscd, amsmath, amssymb, amsfonts, comment}
\usepackage[UKenglish]{babel}
\usepackage[mathcal]{euscript}
\numberwithin{equation}{section}
\newtheorem{theorem}[equation]{Theorem}
\newtheorem{lemma}[equation]{Lemma}

\newtheorem{problem}[equation]{Problem}

\newtheorem{proposition}[equation]{Proposition}

\newtheorem{remark}[equation]{Remark}
\newcommand{\D}{\mathbb{D}}
\newcommand{\C}{\mathbb{C}}
\newcommand{\T}{\mathbb{T}}

\newcommand{\R}{\mathbb{R}}
\newcommand{\N}{\mathbb{N}}

\usepackage{tcolorbox}

        	\definecolor{jazzberryjam}{rgb}{0.65, 0.04, 0.37}

\title[Orthogonal polynomials in $H(b)$]{Orthogonal polynomials in \\ de Branges--Rovnyak spaces}

\author[E. Dellepiane]{Eugenio Dellepiane}
\address{%
D\'{e}partement de math\'{e}matiques et statistique,
Universit\'{e} Laval,
Qu\'{e}bec, QC,
Canada G1K 7P4}
\email{dellepianeeugenio@gmail.com}

\author[D. Seco]{Daniel Seco}
\address{Universidad de la Laguna e IMAULL,  Avenida Astrof\'isico Francisco S\'anchez, s/n, Departamento de An\'alisis Matem\'atico \\ 38206 San Crist\'obal de La Laguna, Santa Cruz de Tenerife,  Spain} \email{dsecofor@ull.edu.es}

\date{\today}

\begin{document}

\begin{abstract} Given a function $b$, holomorphic on the disc and bounded by 1, one can construct an associated reproducing kernel Hilbert space called the de Branges--Rovnyak space $H(b)$. We explore representations of such spaces via descriptions of the corresponding families of orthogonal polynomials. We find relevant structures in the linear systems involved  in a diversity of cases when $b$ is rational. We also establish a form of invariance under some composition operators on $H(b)$ spaces.
\end{abstract}



\thanks{This collaboration was based on a visit of the first author to Universidad de La Laguna, funded by its Vicerrectorate of Research and Transfer. The first author is member of Gruppo Nazionale per l'Analisi Matematica, la Probabilità e le loro Applicazioni (GNAMPA) of Istituto Nazionale di Alta Matematica (INdAM). The second author was supported by Grant PID2024-160185NB-I00 funded by MICIU/AEI/10.13039/501100011033 and by FEDER, UE, through the Generaci\'on de Conocimiento programme of Spanish Ministry of Science, Innovation and Universities;  and by grant RYC2021-034744-I of the Ram\'on y Cajal programme from Agencia Estatal de Investigaci\'on (Spanish Ministry of Science, Innovation and Universities).}

\maketitle

\section{Introduction}

In the context of invariant subspaces for operators, an outstanding role is played by so-called de Branges--Rovnyak spaces $H(b)$. For each $b$ holomorphic over the unit disc and bounded (in modulus) by 1, there is an associated $H(b)$ space, which can be described in terms of its reproducing kernel. Indeed $H(b)$ is the only reproducing kernel Hilbert space over the unit disc with reproducing kernel $k_w$ at $w \in \D$ given by \[k_w(z) = \frac{1-\overline{b(w)}b(z)}{1-\overline{w}z}, \qquad z\in\mathbb{D}.\]
Their nature can be extremely diverse and a wide class of contractive operators on Hilbert spaces are unitarily equivalent to the backward shift's action on different $H(b)$ spaces. However unfortunate, we are still very far from a complete or even satisfactory understanding of their structure. We refer the reader to the standard references \cite{hb1, hb2, sarHb} for general concepts related to this topic. Usually, the $H(b)$ spaces are divided into two main categories, depending on whether the symbol $b$ is an extreme point of the closed unit ball of $H^\infty$, that is often denoted~$H^\infty_1$. Here, \emph{extremality} is intended with respect to the convex structure of $H^\infty_1$, and has an explicit analytic description: $b$ is non-extreme if and only if $\log(1-|b|)\in L^1(\mathbb{T})$. $H(b)$ spaces with non-extreme symbol share more similarities with the Hardy space $H^2$. For example, the shift operator $M_zf(z)=zf(z)$ is well-defined and bounded on $H(b)$ if and only if $b$ is non-extreme. However, the most interesting property for the purpose of this study is the following: $H(b)$ contains the set of polynomials~$\mathcal{P}$ if and only if $b$ is non-extreme. In this case, $\mathcal{P}$ is always dense in $H(b)$. Note that, in general, $H(b)$ spaces with an extreme symbol $b$ may even be finite-dimensional (e.g., whenever $b$ is a finite Blaschke product).

We are interested in understanding the approximation properties of polynomials in these spaces. With this in mind, we set ourselves a clear first goal: obtaining explicit orthonormal bases of polynomials in de Branges--Rovnyak spaces with a non-extreme symbol. In particular, we look for the result of applying the (normalized) Gram-Schmidt method to the canonical basis, leading to sets $\{p_n\}_{n\in\mathbb{N}}$ such that $p_n\in\mathcal{P}_n\setminus\mathcal{P}_{n-1}$ for every $n\in\mathbb{N}$ and $\overline{\operatorname{span}\{p_n\}_n}=~H(b)$, where $\mathcal{P}_n$ denotes the vector space of analytic polynomials with degree at most~$n$. The density of the span of the obtained sequence $\{p_n\}_n$ is deduced from that of $\mathcal{P}$. Of course, not all sets of orthonormal polynomials have necessarily this structure: for example, the set $\{(z+1)/\sqrt{2},(z-1)/\sqrt{2}\}\cup\{z^n\colon n\geq 2\}$ is a complete orthonormal basis in $H^2$. In an ideal situation, one would like to find a complete description of such polynomials. There could be cases in which, via Favard's classical theorem, a three-term recurrence relation for the orthogonal polynomials leads to a spectral measure making $H(b)$ equivalent to a space of the form $\operatorname{Hol}(\D) \cap L^2(\mu)$.  However, it seems too optimistic at this point to expect a full description by means of classical orthogonal polynomials theory.

The norm defined on $H(b)$ spaces can be quite mysterious and hard to work with. Nevertheless, there are some allies. When the function $b$ is non-extreme, there always exists an associated (unique) outer function $a$, called \emph{Pythagorean mate} of $b$. The quotient $\phi:=b/a$ is an analytic function on the unit disc. Instead of working with pairs $(b,a)$, one could equivalently consider functions~$\phi$ in the Smirnov class. All interesting cases of non-extreme de Branges--Rovnyak spaces have in common that $\phi$ has at least one singularity on the boundary $\mathbb{T}:=\partial\mathbb{D}$.

In this article, we describe the orthogonal polynomials for $H(b)$ spaces and present methods to find closed formulas for such polynomials, where $\phi$ is a rational function with  a single pole at a point of the boundary. 
We obtain closed formulas for the polynomials on each particular case by exploiting the structure found in the matrices associated to linear systems that appear in the process. We believe that this structure is a relevant finding in and of itself: the matrices can be transformed, with a few elementary operations, into (fixed and explicit) finite rank perturbations of rather simple banded matrices.

We begin, in Section \ref{S:preliminaries}, by introducing some notation, recalling well established results and we launch our quest by showing that there is some rigidity to what orthogonal polynomial sequences can be made of in $H(b)$ spaces. Indeed, Theorem \ref{T:monomials} shows that $H^2$ is the only $H(b)$ where the set of monomials is orthogonal. 

In Section \ref{S:mainresults}, we focus on functions of the form \begin{equation}\label{eqn1000}
\phi(z) =  A +\frac{B}{1-\overline{\zeta}z},\qquad z\in\mathbb{D}
\end{equation}
where $\zeta\in \T$ and $A$ and $B$ satisfy certain algebraic relation. In particular, as a consequence of Theorem \ref{T:rationalAB}, we discuss the orthogonal polynomials in a de Branges--Rovnyak space that has connections with a celebrated paper of Sarason \cite{Sarason1997LocalDS} and with a recent work of Fricain--Mashreghi \cite{fricainmashreghiorthonormalpoly}. Multiple equidistributed simple poles are solved in the same cases thanks to Theorem \ref{T:z^N}. In particular, we obtain the following explicit example:

\begin{theorem}\label{T: 1+z^N}
    Let $N\in\mathbb{N}\setminus\{0\}$ and $b(z)=(1+z^N)/2,z\in\mathbb{D}$. Then, the associated $H(b)$ space admits the following orthonormal basis of polynomials:
    \begin{align*}
\begin{aligned}
p_j(z)&=\frac{z^j}{\sqrt{2}},  &\qquad  &j<N,   \\
p_{Nk+i}(z)&=\frac{z^{N(k-1)+i}(z^N-1)}{2},  &\qquad &k>0, \quad 0\leq i<N.  
\end{aligned}
\end{align*}
\end{theorem}

We conclude Section \ref{S:mainresults} by discussing the de Branges--Rovnyak space associated to the composition $b(z^N)$, $N\in\mathbb{N}$, for a given symbol $b$. This opens some questions concerning composition operators on $H(b)$. In Section~\ref{S:doublepoles} we complete the picture on orthogonal polynomials by studying, firstly, a general $\phi$ as in~\eqref{eqn1000}, leading to a 3-term recurrence relation satisfied by the coefficients of the orthonormal polynomials. This recurrence relation is presented on our Theorem \ref{T:recurrencerelation}. Then, in Section \ref{higher}, we show with examples how to extend our methods to higher order poles. The polynomials can be obtained by solving certain linear systems, and we conclude our article in Section \ref{conclusion}, proposing a description of the structure of the involved matrices.

\section{Preliminaries and notation}\label{S:preliminaries}

The Hardy space $H^2$ is defined as
\[
H^2=\{f(z)=\sum_{n=0}^\infty a_nz^n \in \operatorname{Hol}(\mathbb{D})\colon \|f\|_2^2=\sum_n|a_n|^2<\infty \}.
\]
It is a Hilbert space with respect to the natural inner product $\langle\cdot,\cdot\rangle_2$. The space $H^\infty$ is the space of bounded analytic functions on $\mathbb{D}$, and it is a Banach space when endowed with the norm $\|f\|_\infty=\sup_{\mathbb{D}}|f|$. For more details, we refer to \cite{garnett}.

We move on to the de Branges--Rovnyak spaces. As mentioned in the Introduction, standard references for these spaces are \cite{hb1,hb2,sarHb}. Inner and outer functions will play a special role in here and we refer to \cite{garnett} for more on them. Given a function $b$ in the closed unit ball of $H^\infty$, that is,
\[
H^\infty_1=\{f\in H^\infty \colon \|f\|_\infty\leq 1\},
\]
the associated Toeplitz operator $T_b\colon H^2\to H^2,f\mapsto bf$ satisfies $I-T_b T_b^\ast\geq 0$, in the sense of operators. The associated de Branges--Rovnyak space $H(b)$ is then defined as the range $(I-T_b T_b^\ast)^\frac{1}{2}H^2$, and it is a Hilbert space endowed with the range norm
\[
\|(I-T_b T_b^\ast)^\frac{1}{2}f\|_{H(b)}:=\|P_{\operatorname{Ker}(I-T_b T_b^\ast)^\frac{1}{2}} f\|_2,\qquad f\in H^2,
\]
where $P_{V}$ stands for the orthogonal projection onto the closed subspace $V$ (in this case, the kernel of $(I-T_b T_b^\ast)^\frac{1}{2}$). This is coherent with the previous description, since $H(b)$ can also be realized as the reproducing kernel Hilbert space associated to the kernel
\[
k_w(z) = \frac{1-\overline{b(w)}b(z)}{1-\overline{w}z}, \qquad z,w\in\mathbb{D}.
\]
Allow us to restate that when $b$ is a non-extreme point of $H^\infty_1$, that is, $\log(1-|b|)\in L^1(\mathbb{T})$, then $H(b)$ contains the set of all polynomials $\mathcal{P}$, and $\mathcal{P}$ is dense in $H(b)$. 

When $b$ is non-extreme, we can define a unique \emph{outer} function~$a$ such that $a(0)>0$ and $|a|^2+|b|^2=1$ $m$-a.e. on~$\mathbb{T}$, as
\[
a(z):= \exp{\left(\int_\mathbb{T} \frac{\zeta+z}{\zeta-z}\log(1-|b(\zeta)|^2)^\frac{1}{2}\,\operatorname{d}\!m(\zeta)\right)}, \qquad z\in\mathbb{D},
\]
where $m$ denotes the normalized Lebesgue measure. We say that $a$ is the Pythagorean mate of $b$, and that $(b,a)$ is a Pythagorean pair. The existence of such function $a\in H^\infty$ that realizes $|a|^2+|b|^2=1$ $m$-a.e. on~$\mathbb{T}$ is actually equivalent (\cite[Theorem 4.1]{garnett}) to the fact that $b$ is non-extreme, for
\[
\int_{\mathbb{T}}\log(1-|b|^2)\operatorname{d}\!m = \int_{\mathbb{T}}\log(|a|^2)\operatorname{d}\!m >-\infty.
\]

Instead of working with pairs $(b,a)$, one could equivalently consider functions~$\phi$ in the Smirnov class. The following result comes from Proposition 3.1 and Remark 3.2 in \cite{Sarason2008unbounded}.
\begin{proposition}
A non-zero function $\phi$ in the Smirnov class can be written uniquely as $\phi=b/a$, where $(b,a)$ is a Pythagorean pair. Furthermore, if the Smirnov function $\phi$ is rational, then the functions $a$ and $b$ are rational.
\end{proposition}
Given a pair $(b,a),$ the quotient $\phi:=b/a$ is an analytic function on the unit disc, and it defines a (possibly unbounded) Toeplitz operator $T_\phi$. It turns out that the domain of the adjoint~$T_{\bar{\phi}}$ is exactly $H(b)$, and that for every $f$ in $H(b)$ the formula
\begin{equation} \label{E:formulanormphi}
\|f\|_b^2=\|f\|_2^2+\|T_{\bar{\phi}} f\|_2^2,
\end{equation}
holds, see \cite{Sarason2008unbounded}. We denote by $\|\cdot\|_b$ and $\langle \cdot,\cdot\rangle_b$, respectively, the norm and the inner product of $H(b)$. 

The identity \eqref{E:formulanormphi} is far from being explicit, but it will allow us to compute the inner products $\langle z^j,z^k\rangle_b$, that are essential bricks in our analysis. In what follows, $(\phi_n)_n$ will denote the Taylor coefficients in the analytic expansion of $\phi$ around $0$, that is,
\begin{equation} \label{E:phicoeff}
\phi(z)=\frac{b(z)}{a(z)}=\sum_{n=0}^\infty \phi_n z^n, \qquad z\in\mathbb{D}.
\end{equation}

The operator $T_{\overline{\phi}}$ acts on functions $f(z)=\sum_{n=0}^\infty a_n z^n$ in $\operatorname{Hol}(\overline{\mathbb{D}})$ as
\[
T_{\overline{\phi}} f(z)=\sum_{m=0}^\infty\left( \sum_{n=0}^\infty\overline{\phi_n}a_{m+n}\right)z^m, \qquad z \in\mathbb{D}.
\]
For more details, see \cite[Section 6]{Sarason2008unbounded}. In particular, for each $k\in\mathbb{N}$ we obtain the following formulas for the monomial $z^k$:
\begin{equation}\label{E:formulaz^n+}
    T_{\overline{\phi}}z^k=\sum_{m=0}^k \overline{\phi_{k-m}}z^m.
\end{equation}

This leads to the the following result, that already appeared in \cite{Chevrot2010}. We include an easy proof for the sake of completeness.
 \begin{lemma}\label{L:formulainnerproduct}
Let $(b,a)$ be a Pythagorean pair, $\phi$ as in \eqref{E:phicoeff} and $j, k \in \N$ with $j\leq k$. Then,
\begin{equation*}
\langle z^j,z^k\rangle_b=\delta_{j,k} + \sum_{s=0}^j \overline{\phi_s}\phi_{k-j+s}.
\end{equation*}
\end{lemma}
\begin{proof}
Equation \eqref{E:formulanormphi} gives
\[\langle z^j,z^k\rangle_b=\delta_{j,k}+\langle T_{\bar{\phi}} z^j,T_{\bar{\phi}}z^k\rangle_2.\]
Using now the explicit expression for $T_{\bar{\phi}} z^j,T_{\bar{\phi}} z^k$ in \eqref{E:formulaz^n+}, we obtain
\[\langle T_{\bar{\phi}} z^j,T_{\bar{\phi}}z^k\rangle_2=\sum_{n=0}^j\sum_{s=0}^k \overline{\phi_{j-n}} \phi_{k-s}\delta_{n,s}=\sum_{s=0}^j \overline{\phi_s}\phi_{k-j+s}. \]
\end{proof}

In many classical spaces of analytic functions, there exists an orthonormal basis of polynomials formed by monomials. We show that this is never the case in non-trivial $H(b)$ spaces. We will use the following well-known fact: $H(b)=H^2$ if and only if $\phi\in H^\infty$. See for example \cite{malmanseco} for relations between $H(b)$ and Hardy spaces.

\begin{theorem}\label{T:monomials}
Let $(b,a)$ be a Pythagorean pair, $\phi$ as in \eqref{E:phicoeff}. If the set of monomials $\{z^n\}_{n\in\mathbb{N}}$ is orthogonal in $H(b)$, then $H(b)=H^2$.
\end{theorem}

\begin{proof}
If $\phi\equiv 0$, then $H(b) = H^2$ and the conclusion holds. Otherwise, let $k_0:=\min\{k\in\mathbb{N}\colon \phi_k\neq 0\}.$ First, we are going to show that $\phi=\phi_{k_0}z^{k_0} (\in H^\infty)$. Indeed, for each $m>0$,
\[
0=\langle z^{k_0},z^{k_0+m}\rangle_b=\sum_{s=0}^{k_0} \overline{\phi_s}\phi_{m+s}=\overline{\phi_{k_0}}\phi_{m+k_0}.
\]
Since $\phi_{k_0}\neq 0$, we proved that $\phi_n=0$ for every $n\neq k_0$.

Now, $\phi$ is a monomial, and thus a bounded function, from which we may conclude that $H(b)=H^2$ concluding the proof.
\end{proof}

To conclude this section, we discuss the Gram-Schmidt algorithm. This is very standard, but we recall some details as a way to fix notation. The notation $f\bot g$ means that $f$ and $g$ are orthogonal in the norm of $H(b)$, that is, $\langle f,g\rangle_b =0$.

For the rest of the paper, we consider orthonormal bases of polynomials $\{p_n\}_{n\in\mathbb{N}}$ such that the degree $\operatorname{deg}(p_n)=n$ for every $n\in\mathbb{N}$. We will write
\[
p_n(z)=\sum_{k=0}^n c_k^{(n)}z^k,\qquad z\in\mathbb{D}.
\] Whenever $n$ can be understood to be fixed, we will drop the superindex $(n)$.
In order to guarantee uniqueness, we choose $c_n^{(n)}>0$. Since $\operatorname{deg}(p_0)=0$, $p_0$ is just a constant and $p_0=1/\|1\|_b$. By induction, for $n\geq 1$, we define $p_n$ in terms of the polynomials $p_k$ with $k<n$ as 
\[p_n(z)=c_n^{(n)}\big(z^n - \sum_{k=0}^{n-1}\langle z^n, p_k\rangle_b p_k\big).\]
 Thus, $c_n^{(n)}>0$ is the unique positive number such that $\|p_n\|_b=1$. 

Notice that, equivalently, the polynomial $p_n$ is the unique polynomial of degree $n$, with norm equal to $1$ and positive $n$-th coefficient that satisfies the $n$ equations given by $\langle p_n,z^k\rangle_b=0,k=0,\ldots,n-1$. When $k=0$, this orthogonality relation follows from the fact that $p_n\bot p_0$ and, for $0<k<n$, one can argue by induction on the degree $k$.

\section{Main results}\label{S:mainresults}

In what follows, we restrict ourselves to the case of a rational $\phi$ with a single pole $\zeta$ on the unit circle $\T$. Notice that the pole of $\phi$ can never lie in $\D$, because $b$ is analytic and $a$ is outer, and the case where the pole lies outside of $\D$ seems less interesting to us, since in that case $\phi$ is in $H^\infty$ and then $H(b)=H^2$. Moreover, by the following argument, we can always assume that $\zeta=1$ in \eqref{eqn1000}.
\begin{proposition}\label{propo1}
    Consider a non-extreme $b$ in $H^\infty_1$, and $\gamma\in\mathbb{R}$. If we consider the rotation $\widetilde{b}(z):=b(e^{i\gamma}z),z\in\mathbb{D}$, then the sequence of orthonormal polynomials $\{\widetilde{p_n}\}_n$ of $H(\widetilde{b})$ is obtained by the sequence of orthonormal polynomials $\{p_n\}_n$ of $H(b)$ via the rotation
    \[
    \widetilde{p_n}(z)=e^{-in\gamma}p_n(e^{i\gamma}z),\qquad z\in\mathbb{D}.
    \]
\end{proposition}
\begin{proof}
    It is immediate to check that $\widetilde{a}(z):=a(e^{i\gamma}z),z\in\mathbb{D},$ and $\widetilde{b}$ form a Pythagorean pair, and $\widetilde{\phi}(z):=\widetilde{b}(z)/\widetilde{a}(z)=\phi(e^{i\gamma}z).$ In particular, the Taylor coefficients of $\widetilde{\phi}$ are given by $\widetilde{\phi}_k = e^{ik\gamma}\phi_k.$ Then, by Lemma \ref{L:formulainnerproduct}, for $j\leq k$
    \begin{align*}
        \langle z^j,z^k\rangle_{\widetilde{b}}=\delta_{j,k} + e^{i(k-j)\gamma}\sum_{s=0}^j \overline{\phi_s}\phi_{k-j+s}=e^{i(k-j)\gamma} \langle z^j,z^k\rangle_{b}.
    \end{align*}
    We write \(p_n(z)=\sum_{j=0}^n c_j^{(n)}z^n\) and we set $q_n(z):=p_n(e^{i\gamma}z), z\in\mathbb{D}.$ For $n\leq m$
    \[
    \langle q_n,q_m\rangle_{\widetilde{b}}=\sum_{j=0}^n\sum_{k=0}^m c_j^{(n)}\overline{c_k^{(m)}}\langle e^{ij\gamma}z^j,e^{ik\gamma}z^k\rangle_{\widetilde{b}}=\langle p_n,p_m\rangle_{b}=\delta_{n,m}.
    \]
    This shows that the polynomials $\{q_n\}_n$ are orthonormal in $H(\widetilde{b})$, and then by uniqueness the polynomials $\{\widetilde{p_n}\}_n$ must have the desired expression, since we want the $n$-th coefficient to be positive.
\end{proof}

\subsection{A first family of examples} We will start considering functions of the form 
\begin{equation} \label{E:phirational}   
\phi(z)=A+\frac{B}{1-z}, \qquad z\in\mathbb{D},
\end{equation}
 where $A,B$ are complex numbers. We will soon show how this comprehends one of the easiest and most studied example of de Branges-Rovnyak spaces. However, an explicit basis of polynomials is still lacking, and as we will see complex phenomena already occur in this basic case. First, we show that a rigid algebraic relation between $A$ and $B$ allows for a very easy set of orthonormal polynomials. 

\begin{theorem}\label{T:rationalAB}
    Let $\phi$ be defined as in \eqref{E:phirational}, with $A,B\in\mathbb{C}$. The set of polynomials $\{q_n\}_{n\in\mathbb{N}}$ with $q_0=1$, $q_n(z)=z^{n-1}(z-1), n>0,$ is orthogonal in $H(b)$ if and only if 
    \(
    \overline{A}B=-(1+|A|^2).
    \)     Moreover, in this case, $A\neq 0$ and
    \[
    \|q_0\|_b=\sqrt{1+\frac{1}{|A|^2}}, \qquad \|q_n\|_b=|A|+\frac{1}{|A|},\quad  n>0.
    \]
\end{theorem}

The orthonormal polynomials $\{p_n\}$ that we referred to can obviously be determined from $\{q_n\}$ dividing each polynomial $q_n$ by its provided norm.

\begin{proof}
The Taylor coefficients of $\phi$ are 
\[
\phi_n=\begin{cases}
    A+B ,& n=0,\\
    B,&n> 0.
\end{cases}
\]
In particular, by Lemma \ref{L:formulainnerproduct}, for $j\leq k$ we have that
\[
\langle z^j,z^k\rangle_b=\delta_{j,k} + \sum_{s=0}^j \overline{\phi_s}\phi_{k-j+s}=\delta_{j,k}+(\overline{A+B})\phi_{k-j} + |B|^2 j.
\]
We derive the formulas
\begin{equation*}
    \langle z^j,z^k\rangle_b= \begin{cases}
    1+|A+B|^2+|B|^2j ,& j=k,\\
   \overline{A}B+|B|^2(1+j),&j<k.
\end{cases}
\end{equation*}
This shows that, whenever $0\leq j<k$,
\begin{equation*}
    \langle z^j,z^{k+1}\rangle_b= \langle z^j,z^k\rangle_b,
\end{equation*}
which is equivalent to saying that $z^j \bot z^{k}(z-1)$. This already shows that $q_j\bot q_k$, whenever $0\leq j<k-1$, without any assumption on the coefficients $A,B$. To conclude, we show that for every $j\in\mathbb{N}$ we have the equivalence
\[q_j\bot q_{j+1}\iff \overline{A}B=-(1+|A|^2).\]
For $j=0$,
\[ \langle 1,z-1\rangle_b=\langle 1,z\rangle_b-\langle 1,1\rangle_b=\overline{A}B-1-|A|^2-2\operatorname{Re}(\overline{A}B).\]
By confronting real and imaginary parts, it is easy to see that $q_0\bot q_1$ if and only if $\overline{A}B=-(1+|A|^2)$. Now, for $j>0$, the orthogonality relation $z^j\bot z^{j+1}(z-1)$ gives 
\begin{align*}
    \langle z^j(z-1),z^{j+1}(z-1)\rangle_b&=\langle z^{j+1},z^{j+2}- z^{j+1}\rangle_b\\
    &=\overline{A}B-1-|A|^2-2\operatorname{Re}(\overline{A}B),
\end{align*}
and we conclude as in the case $j=0$. This concludes the first part of the proof. Now, assuming that $\overline{A}B=-(1+|A|^2)$, one has that $A+B=-1/\overline{A}$, and then $\|1\|_b^2=1+|A|^{-2}$. A direct computation shows that
\[
\|q_n\|_b^2=2+2|A|^2+|B|^2+2\operatorname{Re}(\overline{A}B)=|B|^2=\bigg(|A|+\frac{1}{|A|}\bigg)^2.
\]
\end{proof}

\subsection{Fricain--Mashreghi's work on Sarason's example} As the reader might imagine, Theorem \ref{T:rationalAB} was obtained after a process of trial and error. In any case, we were trying to understand properly an example treated in a classical work by Sarason. In \cite{Sarason1997LocalDS}, he proved that for a specific choice of a rational function $b$, the space $H(b)$ coincides with the local Dirichlet space
\[
\mathcal{D}_1=\{f\in\operatorname{Hol}(\mathbb{D})\colon \mathcal{D}_1(f):=\int_{\mathbb{D}} |f'(z)|^2\frac{1-|z|^2}{|1-z|^2}\operatorname{d}\!A(z)<\infty\},
\]
with an identity between the norms $\|\cdot\|_b=\|\cdot\|_{\mathcal{D}_1}$, where
\[\|f\|_{\mathcal{D}_1}^2:=\|f\|_2^2 + \mathcal{D}_1(f), \qquad f\in\mathcal{D}_1.\]
This was the starting point for a very interesting line of research. Several researchers have worked to understand when equality between $H(b)$ spaces and Dirichlet spaces may occur. It was later proved in \cite{costara2013} that the equality $H(b)=\mathcal{D}_1$ can be achieved when $b$ is the polynomial $b(z)=(1+z)/2$, with equivalence of norms $\|\cdot\|_b\approx \|\cdot\|_{\mathcal{D}_1}$. See also \cite{infinitelysupportedharmonicallyweighted,Chevrot2010,Eugenionext} for more literature on the equality $H(b)=\mathcal{D}_\mu$.

Choosing $b(z)=(1+z)/2$, then $a(z)=(1-z)/2$ and 
\[\phi(z)=\frac{1+z}{1-z}=-1+\frac{2}{1-z}, \qquad z\in\mathbb{D}.\]
Here, $\phi$ satisfies the algebraic relation in Theorem \ref{T:rationalAB}. In particular, the polynomials 
\begin{equation}\label{E:polyb=(1+z)/2}
    p_0=\frac{1}{\sqrt{2}}, \qquad p_n(z)=\frac{z^{n-1}(z-1)}{2},\quad n>0,
\end{equation} form an orthonormal basis for the de Branges--Rovnyak space $H(b)$. This has meaningful connections with a recent work by Fricain--Mashreghi. In \cite{fricainmashreghiorthonormalpoly}, it is shown that the sequence of polynomials given by 
\begin{equation}\label{D:q_n}
    Q_0=1, \quad Q_n(z)=1+(z-1)(a_0+a_1z+\ldots+a_{n-1}z^{n-1}), \quad n>0,
\end{equation}
is an orthogonal basis for $(\mathcal{D}_1,\|\cdot\|_{\mathcal{D}_1})$, where 
\[
a_k=\frac{1}{\sqrt{5}}\bigg(\frac{1+\sqrt{5}}{2}\bigg)^{2k+1}-\frac{1}{\sqrt{5}}\bigg(\frac{1-\sqrt{5}}{2}\bigg)^{2k+1}, \qquad k\geq 1.
\]
It is also shown that $\|Q_n\|_{\mathcal{D}_1}=\sqrt{a_{n}a_{n-1}}$. It is surprising that, endowing the space $H(b)=\mathcal{D}_1$ with the norm $\|\cdot\|_b$, one obtains the much simpler sequence of orthonormal polynomials that is given in \eqref{E:polyb=(1+z)/2}. As a sanity check, we are going to show the following claim:

\begin{proposition} Let $Q_n$ be as in \eqref{D:q_n}. Then,
  \[\|Q_n\|_{\mathcal{D}_1} \approx \|Q_n\|_{b}.\]
\end{proposition}

\begin{proof} 
One can express these polynomials $Q_n,n>0,$ in terms of the orthonormal basis $\{p_k\}_{k\leq n}$ described in \eqref{E:polyb=(1+z)/2} as
\[
Q_n(z)=1+\sum_{k=0}^{n-1}(z-1)a_kz^k=\sqrt{2}p_0(z)+ 2 \sum_{k=1}^n a_{k-1} p_k(z).
\]
Therefore, we can compute the norms $\|Q_n\|_b$ easily as
\begin{align*}
    \|Q_n\|_b^2&=2+4 \sum_{k=1}^n |a_{k-1}|^2 \\
    &= 2+ 4\sum_{k=0}^{n-1} \bigg|\frac{1}{\sqrt{5}}\bigg(\frac{1+\sqrt{5}}{2}\bigg)^{2k+1}-\frac{1}{\sqrt{5}}\bigg(\frac{1-\sqrt{5}}{2}\bigg)^{2k+1}   \bigg|^2.
\end{align*}
Expanding the square, the formula for the sum of geometric sequences yields
\begin{align*}
    \sum_{k=1}^n |a_{k-1}|^2&=\frac{3+\sqrt{5}}{10}\sum_{k=0}^{n-1}\bigg(\frac{1+\sqrt{5}}{2}\bigg)^{4k}+\frac{3-\sqrt{5}}{10}\sum_{k=0}^{n-1}\bigg(\frac{1-\sqrt{5}}{2}\bigg)^{4k}+\frac{2}{5}n\\
    &= \frac{3+\sqrt{5}}{10}\frac{1-\big(\frac{7+3\sqrt{5}}{2}\big)^n}{1-\frac{7+3\sqrt{5}}{2}}+\frac{3-\sqrt{5}}{10}\frac{1-\big(\frac{7-3\sqrt{5}}{2}\big)^n}{1-\frac{7-3\sqrt{5}}{2}}+\frac{2}{5}n.
\end{align*}
In particular, as $n\to\infty,$
\[
\|Q_n\|_b= (1+o(1)) \frac{2\sqrt[4]{5}}{5}\bigg(\frac{3+\sqrt{5}}{2}\bigg)^n,
\]
which is coherent with the fact that 
\[
\|Q_n\|_{\mathcal{D}_1} =(1+o(1)) \frac{1}{\sqrt{5}}\bigg(\frac{3+\sqrt{5}}{2}\bigg)^n,
\]
as found by Fricain--Mashreghi.
\end{proof}

\subsection{Composition with a monomial}
Given a symbol $b$, we study the orthonormal polynomials of the de Branges--Rovnyak space associated to the composition $b(z^N)$, for $N\in\mathbb{N}$. As a byproduct, we complete the example discussed in the previous section and we prove Theorem \ref{T: 1+z^N}, explicitly determining a basis of polynomials for $H(b),$ when $b(z)=(1+z^N)/2$. Moreover, this analysis leads to interesting general questions on composition operators on de Branges--Rovnyak spaces, that we will discuss at the end.

We begin with a general result concerning \emph{inner} functions, that may be already known to experts in the field. However, we could not find a reference. A function $u\in H^\infty$ is inner if $|u|=1$ $m$-a.e. on $\mathbb{T}$.

\begin{proposition}
    Let $b$ be a non-extreme function in $H^\infty_1$ and $a$ its Pythagorean mate. If $u$ is an inner function such that $u(0)=0$, then the composition $b\circ u$ is a non-extreme function of $H^\infty_1$ and its Pythagorean mate is $a\circ u$.
\end{proposition}
\begin{proof}
    That $b\circ u$ is well-defined and in $H^\infty_1$ is trivial. It is also trivial that $a\circ u$ is outer and $a(u(0))=a(0)>0$. That $(b\circ u,a\circ u)$ is a Pythagorean pair is a consequence of the fact that inner functions that vanish at $0$ preserve the Lebesgue measure, that is, $m(E)=m(u^{-1}(E))$ (see for example \cite[page 215]{cima2006cauchy} for a proof that uses Clark measures). Indeed, the set 
    \[
    \{\zeta\in\mathbb{T}\colon |b(u(\zeta))|^2+|a(u(\zeta))|^2=1\}=u^{-1}\left(\{\zeta\in\mathbb{T}\colon |b(\zeta)|^2+|a(\zeta)|^2=1\}\right)
    \]
    has full measure.
\end{proof}
Given $\psi\colon\mathbb{D}\to\mathbb{D}$ analytic, we denote $C_\psi f:=f\circ \psi$ the associated composition operator. Given a Pythagorean pair $(b,a)$, $\phi=b/a$ and a fixed positive natural number $N$, we consider the Pythagorean pair given by $C_{z^N}b(z)=b(z^N)$ and $C_{z^N}a(z)=a(z^N)$. Their quotient is
\[
\frac{b(z^N)}{a(z^N)}=\phi(z^N)=C_{z^N}\phi(z), \qquad z\in \mathbb{D}.
\]
Whenever we talk about \emph{the} orthonormal polynomials, we will mean those obtained by applying the Gram-Schmidt method to the basis of monomials. We are now going to derive the orthonormal polynomials $\{q_n\}_n$ of the space $H(C_{z^N}b)$ in terms of the orthonormal polynomials $\{p_n\}_n$ of the space $H(b)$. 
\begin{theorem}\label{T:z^N}
    Let $(b,a)$ be a Pythagorean pair, $\phi=b/a$,$\{p_n\}_n$ the set of orthonormal polynomials in $H(b)$, $N\in\mathbb{N}, N>0,$ as above. Consider the orthonormal basis $\{q_n\}_n$  of polynomials of $H(C_{z^N}b)$. Then, for every $l\in\mathbb{N}$ and for every $0\leq i<N$, we have that 
    \[q_{Nl+i}(z)=z^ip_l(z^N), \qquad z\in\mathbb{D}.\]
\end{theorem}
\begin{proof}
   For simplicity, we set $\widetilde{\phi}:=C_{z^N}\phi, \widetilde{b}:=C_{z^N}b$. If the Taylor coefficients of $\phi$ are given by $\phi(z) = \sum_{k=0}^\infty \phi_k z^{k}$, then by the formula $\widetilde{\phi}(z)= \sum_{k=0}^\infty \phi_k z^{kN}$ we know that the Taylor coefficients of $\widetilde{\phi}$ are
\[
    \widetilde{\phi}_k=\begin{cases} 0 ,& k\not\equiv 0 \mod N,\\
    \phi_{\frac{k}{N}},& k \equiv 0 \mod N.
\end{cases}
\]
Given a real number $x$, we denote its integer part by
\(\lfloor x \rfloor :=\max\{k\in\mathbb{Z}\colon k\leq x\}\). By Lemma \ref{L:formulainnerproduct}, for $j,k \in \N$, $j\leq k$, we have that
\begin{align*}
\langle z^j,z^k\rangle_{\widetilde{b}}&
= \delta_{j,k} + \sum_{s=0}^{j}  \overline{\widetilde{\phi}_s}\widetilde{\phi}_{k-j+s}\\
&=\delta_{j,k} +\sum_{s=0}^{N \lfloor j/N \rfloor} \overline{\widetilde{\phi}_s}\widetilde{\phi}_{k-j+s}\\
&=\delta_{j,k} +\sum_{s=0}^{\lfloor j/N  \rfloor}\overline{\widetilde{\phi}_{sN}}\widetilde{\phi}_{k-j+sN}.
\end{align*}
In particular, note that if $k\not \equiv j\mod N,$ then $\langle z^j,z^k\rangle_{\widetilde{b}}=0$, as $\widetilde{\phi}_{k-j+sN}=0$ for every $s$. Assuming that $k \equiv j \mod N,$ it is more convenient to calculate the inner products of $z^{Nj+i}$ and $z^{Nk+i}$, where $j\leq k$ are natural numbers and $0\leq i<N$:
\begin{align*}
\langle z^{Nj+i},z^{Nk+i}\rangle_{\widetilde{b}}&
= \delta_{Nj+i,Nk+i} + \sum_{s=0}^{j}\overline{\widetilde{\phi}_{sN}}\widetilde{\phi}_{Nk-Nj+sN}\\
&= \delta_{j,k} +\sum_{s=0}^{j}\overline{\phi_{s}}\phi_{k-j+s}\\
&=\langle z^j,z^k\rangle_{b}.
\end{align*} 
This formula holds independently of $0\leq i<N$. Using this, now we show that $q_{Nl+i}(z)=z^ip_l(z^N)$ for every $l\in\mathbb{N}$ and $0\leq i<N$, by showing that $z^ip_l(z^N)$ is orthogonal to every monomial $z^s$ with $s<Nl+i$. Let 
\[p_l(z)=\sum_{k=0}^l c_k^{(l)}z^k,\qquad z\in\mathbb{D}.\] We have that
\[
\langle z^ip_l(z^N),z^s\rangle_{\widetilde{b}}=\sum_{k=0}^l c_k^{(l)}\langle z^{Nk+i},z^s\rangle_{\widetilde{b}}=0
\]
whenever $s\not \equiv i \mod N$. If else $s=Nm+i$ with $m<l$, then
\begin{align*}
\langle z^ip_l(z^N),z^{Nm+i}\rangle_{\widetilde{b}} &=\sum_{k=0}^l c_k^{(l)}\langle z^{Nk+i},z^{Nm+i}\rangle_{\widetilde{b}} \\ &= \sum_{k=0}^l c_k^{(l)}\langle z^{k},z^{m}\rangle_{b}=\langle p_l,z^m\rangle_b=0.
\end{align*}
Notice also that 
\begin{align*}
\|z^ip_l(z^N)\|_{\widetilde{b}}^2 &=\sum_{k,j=0}^l c_k^{(l)}\overline{c_j^{(l)}}\langle z^{Nk+i},z^{Nj+i}\rangle_{\widetilde{b}} \\ &= \sum_{k,j=0}^l c_k^{(l)}\overline{c_j^{(l)}}\langle z^{k},z^{j}\rangle_{b}=\|p_l\|_{b}^2=1.
\end{align*}
\end{proof}

As a corollary, we immediately deduce Theorem \ref{T: 1+z^N}.

\begin{proof}[Proof of Theorem \ref{T: 1+z^N}]
    The proof follows from \eqref{E:polyb=(1+z)/2} and Theorem \ref{T:z^N}.
\end{proof}

As a consequence of Theorem \ref{T:z^N}, we also obtain the following result, concerning the structure of the de Branges--Rovnyak space $H(C_{z^N}b)$.  
\begin{theorem}\label{T:H(b(z^N))}
    Let $b\in H^\infty_1$ non-extreme and $N\in\mathbb{N}$. Then, the operator $C_{z^N}\colon H(b)\to H(C_{z^N}b), C_{z^N}f(z)=f(z^N),z\in\mathbb{D},$ is a linear isometry. Moreover,
    \[
    H(C_{z^N}b)=\bigoplus_{j=0}^{N-1} M_z^jC_{z^N}H(b),
     \]
    where $M_z$ denotes the shift operator. The direct sum is intended with respect to the inner product of $H(C_{z^N}b)$. 
\end{theorem}
\proof Again, we introduce the notation $\widetilde{b}(z):=b(z^N),z\in\mathbb{D}.$
    First, we show that $C_{z^N}H(b)\subset H(\widetilde{b})$. Given $h\in H(b)$, there exist coefficients $(a_n)_{n\in\mathbb{N}}$ such that $h=\sum_n a_n p_n,$ where $\{p_n\}_n$ is the orthonormal basis of polynomials of $H(b)$. Such series converges in $H(b)$ and $\sum_n|a_n|^2=\|f\|_b^2<\infty$. Since pointwise evaluations are continuous operators on $H(b)$, we have that
    \[
    h(z^N)=\sum_n a_n p_n(z^N)=\sum_n a_n q_{nN}(z),\qquad z\in\mathbb{D},
    \]
    where the polynomials $\{q_k\}_k$ are defined as in Theorem \ref{T:z^N}. In particular, $C_{z^N}b\in H(\widetilde{b})$ and
    \[
    \|C_{z^N}h\|_{\widetilde{b}}^2=\sum_n|a_n|^2=\|h\|_b^2.
    \]
    This shows that $C_{z^N}H(b)\subset H(\widetilde{b})$ and it is a closed subspace in the norm topology of $H(\widetilde{b})$. Since the shift operator is well-defined on de Branges--Rovnyak spaces with non-extreme symbols, we have that 
    \[H(\widetilde{b})\supset \bigcup_{j=0}^{N-1} M_z^j C_{z^N}H(b).\]
    Next, we show that $M_{z}^{j_1}C_{z^N}H(b) \bot M_{z}^{j_2}C_{z^N}H(b),$ for $0\leq j_1,j_2<N, j_1\neq j_2$. For $i=1,2$, let $f_i\in M_z^{j_i}C_{z^N}H(b).$ There exist functions $h_i\in H(b),i=1,2$, such that
    \[
    f_i(z)=z^{j_i}h_i(z^N), \qquad z\in\mathbb{D}, \quad i=1,2.
    \]
    Representing $h_i$ with respect to the orthonormal basis of $H(b)$, there also exist coefficients $(a_n^{(i)})_{n\in\mathbb{N}}$ such that
     \[
    f_i(z)=z^{j_i}\sum_n a_n^{(i)}p_n(z^N), \qquad z\in\mathbb{D}, \quad i=1,2.
    \]
    Then,
    \begin{align*}
        \langle f_1,f_2\rangle_{\widetilde{b}}&=\sum_{k,n=0}^\infty a_n^{(1)}\overline{a_n^{(2)}}\langle z^{j_1}p_n(z^N),z^{j_2}p_n(z^N)\rangle_{\widetilde{b}}=0,
    \end{align*}
    by Theorem \ref{T: 1+z^N}. To conclude, we have to show that
     \[H(\widetilde{b})\subset \bigoplus_{j=0}^N M_z^jC_{z^N}H(b).\]
     For $f\in H(\widetilde{b})$, we can write the Fourier basis associated to the orthonormal polynomials $\{q_k\}_k$ as $f=\sum_n F_n q_n.$ In particular, by linearity,
     \begin{align*}
         f(z)&=\sum_{l=0}^\infty\sum_{j=0}^{N-1}F_{Nl+i}q_{Nl+i}(z)=\sum_{j=0}^{N-1}z^j\sum_{l=0}^\infty F_{Nl+j}p_l(z^N), \qquad z\in\mathbb{D}.
     \end{align*}
    Note that for $0\leq j<N$ the function 
    \[
    g_j(z)=\sum_{l=0}^\infty F_{Nl+j}p_l(z), \qquad z\in\mathbb{D},
    \]
    belongs to $H(b),$ since 
    \[
    \sum_{l=0}^\infty|F_{Nl+j}|^2\leq \sum_{n=0}^\infty|F_n|^2=\|f\|_b^{2}, \qquad 0\leq j<N,   
    \]
    and $\{p_l\}_{l}$ is an orthonormal basis of $H(b)$. We conclude that $f(z)=\sum_{j=0}^{N-1}z^jg_j(z^N),z\in\mathbb{D}$, with $g_j\in H(b),$ thus
     \[
    H(C_{z^N}b)=\bigoplus_{j=0}^{N-1} M_z^j\{C_{z^N}f\colon f\in H(b)\}.\qedhere
    \]

We remark that, in general, it is not trivial to establish whether a function~$h$ belongs to a certain $H(b)$ space. In the previous proof, we showed that $C_{z^N}h\in H(\widetilde{b})$ and then that $g_j\in H(b),j=1,2,$ only using the information gathered on the orthonormal polynomials of de Branges--Rovnyak spaces in Theorem \ref{T:z^N}.

\subsection{Further remarks}

Theorem \ref{T:H(b(z^N))} and its proof raise some interesting questions: 

\begin{problem}
Given a more general function $\psi\colon \mathbb{D}\to\mathbb{D}$, and $b \in H^\infty_1$, what can be said about the de Branges--Rovnyak space $H(C_\psi b)$ and its relation with $H(b)$? \end{problem}

Our analysis provides rather precise answers in the case where $\psi(z)=z^N$ (and $b$ is non-extreme), but this matter can clearly be studied in greater generality. We propose two further problems. The first one is about the bounded embedding: 
\begin{problem}
Is it always true that
\(
C_\psi H(b)\subseteq H(C_\psi b) 
\)
and if so, is
\[
C_\psi\colon H(b)\to H(C_\psi b)
\]
a bounded operator? 
\end{problem}

The second is perhaps even more ambitious:

\begin{problem}
Is there a way to express $H(C_\psi b)$ in terms of $H(b)$?  
\end{problem}

Similar questions are addressed in the context of model spaces, when both $b$ and $\psi$ are inner, in Chapter $6$ of \cite{model} and references therein.

\section{Recurrences and underlying matrix structure} \label{S:doublepoles}

\subsection{Poles of degree 1 in general position} \label{S:phicomplete} 
The restriction on the values of $A$ and $B$ in Theorem \ref{T:rationalAB} seems an unnatural place to stop, because a very simple generalization of the Sarason space leads to a different rational $\phi$ of degree 1. Instead of considering $b=(1+z)/2$, we consider $b=(1+B)/2$, where $B$ is a Blaschke factor having the real zero $c\in(-1,1)$, that is,
\[
B(z)=\frac{z-c}{1-cz}, \qquad z\in\mathbb{D}.
\]
It is perhaps more common to consider as a Blaschke factor $-B$, but our choice turns out to be more convenient for our purposes. In this case, $a=(1-B)/2$ and
\[
\phi(z)=\frac{1+\frac{z-c}{1-cz}}{1-\frac{z-c}{1-cz}}=\frac{1-c}{1+c}\bigg(-1+\frac{2}{1-z}\bigg), \qquad z\in\mathbb{D}.
\]
Notice that this $\phi$ satisfies the hypothesis of Theorem \ref{T:rationalAB} if and only if 
\[
2\frac{1-c}{1+c}=\bigg(1+\frac{(1-c)^2}{(1+c)^2}\bigg)\frac{1+c}{1-c},
\]
which happens if and only if $c=0$. That happens precisely when $b=(1+z)/2$. In particular, this reflects the fact that the condition of Theorem \ref{T:rationalAB} is not preserved when multiplying $\phi$ by a constant. To discuss this example, we consider the more general function \eqref{eqn1000}.
We are going to show that the structure of the orthonormal polynomials of $H(b)$ is dramatically different from the easy case that we already discussed.

\begin{theorem}\label{T:recurrencerelation}
Let $A, B \in \C$, $B \neq 0$, and $\phi$ as in \eqref{eqn1000} with $\zeta=1$. Let $\{p_n\}_{n\in \N}$ be the associated orthonormal basis for $H(b)$, fix $n\in \N$, $n \geq 2$, and denote $p_n(z)= \sum_{k=0}^n c_k z^k$, with $c_n >0$. Then $\{c_k\}_{k=0}^{n-1}$ satisfy the 3-term recurrence relation 
\[\overline{T_0} c_{k+1} + (T_1-T_0) c_k + T_0 c_{k-1}=0, \qquad k = 1, \dots, n-2,\]
where $T_0= 1+|A|^2 +A\overline{B}$ and $T_1= -\overline{T_0}-|B|^2$.
\end{theorem}

Before the proof, we highlight that one simple operation such as scaling $\phi$ by a constant factor, results in a significant increase of complexity in the expression of a basis of orthonormal polynomials. Notice in particular that the orthogonal polynomials when $b(z)=(1+z)/2$ had (at most) 2 non-zero coefficients, where any other space in this family yields polynomials of arbitrary length.

After the proof of the theorem, we will show how to obtain a complete description of the families of polynomials $\{p_n\}_{n\in\mathbb{N}}$ in terms of $A, B$ and $n$ exclusively. However, we recommend to keep the following in mind:
\begin{remark}\label{remark100}
 The coefficients of $p_n$ in Theorem \ref{T:recurrencerelation} cannot be expressed in a short unified manner since their form depends on whether the polynomial \begin{equation}\label{eqnQ}
Q(z) =\overline{T_0} z^2 + (T_1-T_0) z + T_0 
\end{equation} has 
\begin{itemize}
    \item[(i)] degree 1, meaning $A\overline{B}=-(1+|A|^2)$, dealt with in Theorem \ref{T:rationalAB};
    \item[(ii)] two simple roots $\lambda_1$ and $\lambda_2= \frac{T_0}{\overline{T_0}\lambda_1}$ in $\C\backslash\{0,1\}$;
    \item[(iii)] a double root $\lambda = \frac{T_0}{|T_0|} \in \T \backslash \{1\}$.
\end{itemize}
\end{remark}

Now, we are ready to proceed with the proof.

\begin{proof}[Proof of Theorem \ref{T:recurrencerelation}]
The Taylor coefficients of $\phi$ are 
\[
\phi_n=\begin{cases}
    A+B,& n=0,\\
  B,&n> 0.
\end{cases}
\]
An application of Lemma \ref{L:formulainnerproduct} shows that
\begin{equation*}
    \langle z^j,z^k\rangle_b= \begin{cases}
    1 + |A+B|^2 + |B|^2k,& j=k,\\
    \overline{A}B + (j+1)|B|^2,&j<k, \\
        A\overline{B} + (k+1)|B|^2,&j>k. \\
\end{cases}
\end{equation*}
For the sake of convenience, we introduce the notation $\rho:=1+|A+B|^2$. The orthogonality conditions $\langle p_n,z^k\rangle_b=0, k=0,\ldots, n-1,$ define $n$ equations for our $n+1$ variables, that give relations between the coefficients $c_0, \ldots, c_{n}$ of $p_n$. 
For $0 \leq k\leq n-1$, we have that
\begin{align*}
\langle p_n,z^k\rangle_b &= \sum_{j=0}^{k-1} c_j\big(\overline{A}B+|B|^2(1+j)\big) + c_k (\rho+|B|^2k) +\\
&\qquad +\big(A\overline{B}+|B|^2(1+k)\big)\sum_{j=k+1}^n c_j=0.
\end{align*}

On the other hand, notice that $1=\|p_n\|_b^2=c_n\langle p_n,z^n\rangle_b$ (where we used that $c_n>0$), and so if we introduce the mute variable $t=1/c_n (\in \R^+)$ we obtain \begin{align*}
\langle p_n,z^{n}\rangle_b &= \sum_{j=0}^{n-1} c_j\big(\overline{A}B+|B|^2(1+j)\big)+ (\rho+n|B|^2) c_n=t.
\end{align*}
 We find that the vector $(c_0,\ldots, c_{n})$ solves the $(n+1)\times(n+1)$ system represented by the augmented matrix
\[
\left(\begin{array}{ccccc|c}
\rho & A\overline{B}+|B|^2 & A\overline{B}+|B|^2 &  \cdots & A\overline{B}+|B|^2 & 0\\
\overline{A}B+|B|^2 & \rho+|B|^2  & A\overline{B}+2|B|^2 &  \cdots & A\overline{B}+2|B|^2 & 0\\
\overline{A}B+|B|^2 & \overline{A}B+2|B|^2 &  \rho+2|B|^2  &  \cdots & A\overline{B}+3|B|^2 & 0 \\
\vdots & \vdots & \vdots & \ddots &\vdots & \vdots\\
\overline{A}B+|B|^2 & \overline{A}B+2|B|^2 & \overline{A}B+3|B|^2 & \cdots &  \rho+n|B|^2 & t
\end{array} \right).
\]
Denote
\begin{align*}
    T_0&:=\rho-\overline{A}B-|B|^2=1+|A|^2 +A\overline{B},\\
    T_1&:=A\overline{B}-\rho=-\overline{T_0}-|B|^2.
\end{align*}
We perform the following elementary operations. For $j=0,\ldots,n-1$, substitute the row $R_j$ with $R_j-R_{j+1}.$ The first $n$ equations then become
\begin{align*}
   T_0c_k+T_1c_{k+1}-|B|^2\sum_{j=k+2}^n c_j=0, &\qquad k=0,\ldots,n-2\\
   T_0c_{n-1}+T_1 c_n=-t, &\qquad k=n-1,
\end{align*}
and the system is expressed via the matrix
\begin{equation}\label{E:matrix1}
\left(\begin{array}{ccccc|c}
T_0 & T_1 & -|B|^2 &  \cdots & -|B|^2 & 0\\
0 & T_0  & T_1 &  \cdots  & -|B|^2   & 0\\
0 & 0 &  T_0  &  \cdots & -|B|^2  & 0 \\
\vdots & \vdots & \vdots & \ddots &\vdots & \vdots\\
0 & 0 &  0  &  \cdots  & T_1 & -t \\
\overline{A}B+|B|^2 & \overline{A}B+2|B|^2 & \overline{A}B+3|B|^2 & \cdots &     \rho+n|B|^2 & t
\end{array} \right).
\end{equation}
Now, substituting once again $R_j$ with $R_j-R_{j+1}$ for $0\leq j\leq n-3,$ since $T_1+|B|^2=-\overline{T_0}$ we obtain the $n-2$ equations
\[
\left(\begin{array}{cccccccc|c}
T_0 & T_1-T_0 & \overline{T_0} &  0 &\cdots & 0 & 0 & 0& 0\\
0 & T_0  & T_1-T_0 &  \overline{T_0}   &\cdots  & 0 & 0 & 0& 0\\
0 & 0 &  T_0  &  T_1-T_0 & \cdots & 0 & 0  & 0& 0 \\
\vdots & \vdots & \vdots & \vdots & \ddots &\vdots  &\vdots & \vdots&\vdots\\
0 & 0 &  0  & 0 & \cdots  & \overline{T_0} & 0 & 0 & 0\\
0 & 0 &  0  & 0 & \cdots &T_1-T_0 & \overline{T_0} & 0& 0
\end{array} \right),
\]
that define the $3$-term recurrence relation in the statement of the theorem. \end{proof}

Even if the proof is complete, we want to further simplify the linear system~\eqref{E:matrix1}, for which we keep on using all the notation in the proof. The solution to this system can be determined in terms of the roots of the polynomial $Q(z)=\overline{T_0}z^2+(T_1-T_0)z+T_0.$  First, we update the last row $R_n$, replacing it with $R_n + \sum_{j=0}^{n-1} R_j$ and after that, we substitute $R_j$ by $R_j-R_{j+1},j=0,\ldots,n-2.$ Since $T_1+|B|^2=-\overline{T_0}$ we obtain the augmented matrix
\[
\left(\begin{array}{cccccccc|c}
T_0 & T_1-T_0 & \overline{T_0} & 0 &  \cdots & 0 & 0 & 0 & 0\\
0 & T_0  & T_1-T_0 & \overline{T_0} & \cdots & 0 & 0 & 0 & 0\\
0 & 0 &  T_0  & T_1-T_0 & \cdots & 0 & 0  & 0 & 0 \\
\vdots & \vdots & \vdots & \vdots &\ddots & \vdots & \vdots &\vdots & \vdots\\
0 & 0 & 0 & 0 &\cdots  & T_1-T_0 & \overline{T_0} &0 & 0\\
0 & 0 & 0 & 0 &\cdots & T_0 & T_1-T_0 &\overline{T_0} & t\\
0 & 0 &  0  & 0 & \cdots & 0  & T_0 & T_1 & -t \\ \rho & T_2 & T_2  & T_2 & \cdots & T_2  & T_2 & T_2 & 0
\end{array}\right),
\] with \begin{equation}\label{eqn5001}
T_2= \rho - \overline{T_0}= \overline{A}B+|B|^2. \end{equation} 

We can use the fact that $T_0 + (T_1-T_0) + \overline{T_0}= -|B|^2$ to design one more elementary operation that will allow us to get rid of most of the appearances of $T_2$: we replace $R_n$ with $R_n+\frac{T_2}{|B|^2} \sum_{j=0}^{n-1}R_j$. This gives us
\begin{equation}\label{eqn6001}
\left(\begin{array}{cccccc|c}
T_0 & T_1-T_0  & \overline{T_0} & 0 & \cdots & 0  & 0 \\
0 & T_0 & T_1-T_0  & \overline{T_0} & \cdots & 0  & 0 \\
\vdots & \ddots & \ddots &\ddots & \ddots & \vdots  &  \vdots \\
0 & \cdots   & T_0 & T_1-T_0  & \overline{T_0} & 0  & 0 \\
0 & \cdots & 0  & T_0 & T_1- T_0 & \overline{T_0}  & t \\
0 & \cdots & 0 & 0  & T_0 & T_1 & -t \\
T_3 & T_4 & 0 & \cdots   & 0 & 0 & 0
\end{array}\right),
\end{equation} where \begin{equation}\label{eqn5002}
T_3= \rho +  \frac{T_0T_2}{|B|^2}; \qquad T_4 = -\frac{\overline{T_0}T_2}{|B|^2}.
\end{equation}

As mentioned in Remark \ref{remark100}, we can divide the possibilities into three cases:

Case $(i)$: $T_0=0$ or, equivalently, $A\overline{B}=-(1+|A|^2)$. This case is the exact content of Theorem \ref{T:rationalAB}, and we know that $c_k=0$ for each $k\leq n-2$. Outside of case $(i)$, the polynomial $p_n$ is more complicated: if $c_0=c_1=0$, then by the recurrence relation we obtain that $c_k=0$ for every $k\leq n-1$, since $T_0\neq 0$. But then, $p_n$ would be a monomial, since only $c_n\neq 0$, and this is ruled out by Theorem \ref{T:monomials}.

Indeed, if $T_0\neq 0$, $Q$ is a polynomial of degree $2$ that does not vanish at $0$ (since $Q(0)=T_0\neq 0$). The exact solution to \eqref{eqn6001} depends on whether $Q$ has a double root or two simple roots.

Case $(ii)$: $Q$ has two simple non-zero solutions $\lambda_1\neq\lambda_2= \frac{T_0}{\overline{T_0}\lambda_1}.$ One can check that the solution of the recurrence relation in rows $R_0,\ldots, R_{n-3},$ of the system \eqref{eqn6001} is given by 
\[
c_k=\frac{c_1v_k-c_0\lambda_1\lambda_2v_{k-1}}{v_1},
\]
where $v_j=\lambda_2^j-\lambda_1^j$.
This can be used to translate the last 3 rows of \eqref{eqn6001} into the following augmented $3 \times 3$ system on $c_0$, $c_1$, and $c_n$:
\begin{equation}\label{eqn7001}
\left(\begin{array}{ccc|c}
\frac{\lambda_1 \lambda_2(v_{n-2}(T_0-T_1) - v_{n-3}T_0)}{v_1} & \frac{v_{n-2}T_0 + (T_1-T_0)v_{n-1}}{v_1} & \overline{T_0} & t \\
\frac{-T_0 \lambda_1 \lambda_2 v_{n-2}}{v_1} & \frac{T_0 v_{n-1}}{v_1} & T_1  & -t \\
T_3 &  T_4 & 0 & 0
\end{array}\right).
\end{equation}
The solution of this system for $t=1$ will be of the form $(u_0,u_1,u_n)$, and the homogeneity in $t$ implies that $c_n=tu_n$. Since $c_nt= t^2 u_n=1$, we can solve for $t$. Necessarily one of the two solutions will be positive and that is $t$ that leads to the full solution $(c_0, c_1, c_n)$ and thus to a fully closed formula for the polynomial $p_n$.

Case $(iii)$: $Q$ has a double root $\lambda\neq 0$. We can write\[
Q(z)=\overline{T_0}z^2+(T_1-T_0)z+T_0=\overline{T_0}(z-\lambda)^2.
\]
Using that $T_0\neq 0$, we obtain that $\lambda=\frac{T_0}{|T_0|}$. In particular,  $|\lambda|=1$. Moreover, $\lambda$ also has the explicit expressions
\begin{equation}\label{eqn6003}
\lambda=\frac{T_0-T_1}{2\overline{T_0}}=\frac{2T_0}{T_0-T_1}.
\end{equation}
The solution to the recurrence relation given by rows $R_0,\ldots,R_{n-3}$ is 
\[
c_k=\big((1-k)c_0+k\overline{\lambda}c_1\big)\lambda^k, \qquad k=0,\ldots,n-1.
\]
We can use \eqref{eqn6003} to see that $T_2 \neq 0$ (otherwise $T_0=1$ and then $|\lambda|>1$), and thus $T_4 \neq 0$. The equation from $R_{n}$ can thus be expressed as \[c_1=-\frac{T_3}{T_4}c_0\]
Evaluating $c_{n-2}$ and $c_{n-1}$ in terms of $c_0$ and $c_1$, we can translate the remaining 2 equations on \eqref{eqn6001} in terms of $c_0$, $c_n$ and $t$ only. Using \eqref{eqn6003}, row $R_{n-2}$ becomes 
\begin{equation}\label{eqn6006}((n-1) \lambda + n  \frac{T_3}{T_4}) \lambda^{n-1} c_0 + c_n = \frac{t\lambda}{|T_0|}. \end{equation}

The remaining equation $R_{n-1}$ is equivalent to \begin{equation}\label{eqn6007}((2-n)-\overline{\lambda}\frac{T_3}{T_4} (n-1)) \lambda^{n} c_0 + (\lambda-2) c_n = \frac{-t}{|T_0|}.\end{equation}
Solving the $2 \times 2$ linear system given by \eqref{eqn6006} and \eqref{eqn6007} will provide the solution in terms of $t$. If the system is solvable with $t=1$, the general solution is the result of multiplying that one by $t$, since the system is 1-homogeneous on $t$. The expression for $c_n$ will lead to two possible solutions for the equation $c_nt=1$, exactly one of which will give rise to a $c_n>0$, thus representing the complete solution.

\subsection{Higher order poles}\label{higher}
We wonder what happens when $\phi$ has a higher-order singularity at the boundary. Denote by $c$ the vector of coefficients $\{c_k\}_{k=0}^n$ of the polynomial $p_n$, so that $c_n>0$ and $\{p_n\}$ is an orthonormal basis of $H(b)$, with $\phi=b/a$ as before. Let us show with an example what the general principle will be for finite expressions of the form \begin{equation}\label{eqn80}
\phi(z) = r_0 + \frac{r_1}{1-z} + \frac{r_2}{(1-z)^2} = r_0 + \sum_{k=0}^\infty (r_1+r_2(k+1)) z^k.
\end{equation}
Applying Lemma \ref{L:formulainnerproduct}, we can obtain \[M_{j,k}:= \langle z^j, z^k\rangle_b\]
in terms of $r_0$, $r_1$ and $r_2$, for $j,k=0,\dots, n$. For simplicity, let us take $r_0, r_1$ and $r_2$ to be real. This is not an essential issue, but that way the Gram matrix $M=(M_{j,k})_{j,k}\in\mathbb{C}^{(n+1)\times(n+1)}$ becomes symmetric (otherwise, in what follows, think of $M^*$ instead of $M$). The equations $\left \langle p_n, z^k\right\rangle_b = 0$ (that hold for $k=0, \dots, n-1$) can be written as $Mc=t e_n$ where $t$ is a mute variable that can be chosen so that $c_nt=1$ and $e_n$ is the last element of the canonical basis in $\C^{n+1}$.
From now on, we denote by $I$ the identity matrix and by $B$ the backward shift matrix (a matrix with ones on the diagonal immediately above the main one, and zero everywhere else). 

The following proposition contains the most relevant structure on $M$ we were able to find: 
\begin{proposition}\label{propo100}
There exists an upper triangular pentadiagonal matrix $T \in \C^{(n+1) \times (n+1)}$ and a (rank 5) matrix $N$ that starts with $n-4$ null rows such that $(I-B)^4M=T+N$. \end{proposition}

\begin{proof}
Denote $M_{j,k,0}:=M_{j,k}$ and $M_{j,k,d+1}=M_{j,k,d}-M_{j,k+1,d}$, which represents $d+1$ iterations of the elementary row operation on $M$ consisting of replacing row $R_k$ with itself minus row $R_{k+1}$. This is well defined as long as $k+d+1 \leq n$. 

Take, firstly, $j < k \leq n-2$. Then Lemma \ref{L:formulainnerproduct} yields 
\begin{align*}M_{j,k,1} &= M_{j,k}-M_{j,k+1} = \sum_{s=0}^j \phi_s(\phi_{k-j+s}-\phi_{k-j+s+1}) \\ &= \left(\sum_{s=0}^j \phi_s \right) r_2.\end{align*}
Since this does not depend on $k$, iterating one more time, we obtain
\[M_{j,k,2} = M_{j,k,1}-M_{j,k+1,1} = 0.\]
If, on the opposite extreme, $k+5 \leq j \leq n$, we obtain 
\begin{align*} M_{j,k,1} &= \sum_{s=0}^k \phi_s \phi_{j-k+s}- \sum_{s=0}^{k+1} \phi_s \phi_{j-k-1+s} \\ &= r_2 \left(\sum_{s=0}^k \phi_s \right)  - \phi_{k+1}\phi_j.  \end{align*}
We can continue iterating to obtain 
\begin{align*} M_{j,k,2} &= M_{j,k,1}-M_{j,k+1,1} =  -r_2 \phi_{k+1}  - \phi_{k+1}\phi_j + \phi_{k+2}\phi_j \\
&= r_2 (\phi_j-\phi_{k+1}) = r_2^2 (j-k-1) .  \end{align*}
The next iteration takes us to 
\[M_{j,k,3} = M_{j,k,2}-M_{j,k+1,2} = r_2^2 (j-k-1) - r_2^2 (j-k-2) =r_2^2. \]
This is a constant, and hence $M_{j,k,4}=0$.
At this point, we have performed the elementary operations corresponding to multiplying on the left by $(I-B)^4$ the matrix $M$, except for the values in the 5 diagonals where $j-k=0,\dots, 4$, that determine $T$. \end{proof}

\begin{remark} In the notation of Proposition \ref{propo100}, it can be shown that the entries $T_{k,k}$ and $T_{k+4,k}$ are independent of $k$ and $T_{k+1,k}$, $T_{k+2,k}$, $T_{k+3,k}$ are each determined by a polynomial of degree 2 on $k$. Assume that $k \leq n-5$, since the last 5 rows determine the matrix $N$. The general formula for $T_{j,k}$ is \[T_{j,k}= M_{j,k}-4M_{j,k+1}+6M_{j,k+2}-4 M_{j,k+3} + M_{j,k+4}.\] For instance, if $j =k$, we obtain 
\begin{align*}T_{k,k} &= \sum_{s=0}^k \phi_s (\phi_s-4\phi_{s+1}+6\phi_{s+2}- 4\phi_{s+3}+\phi_{s+4}) \\ &= \phi_0 r_0 = r_0(r_0+r_1+r_2),\end{align*} which is a constant independent of $k$.
However, if we take $j =k+1$, we obtain 
\begin{align*}T_{k+1,k} &= - \phi_{k+1}^2 + \sum_{s=0}^{k+1} \phi_s (\phi_s-4\phi_{s+1}+6\phi_{s+2}- 4\phi_{s+3}+\phi_{s+4}) \\ &= - \phi_{k+1}^2 + r_0(r_0+r_1+r_2),\end{align*} which is indeed a polynomial of degree 2 on $k$.
\end{remark}

\subsection{A potential description of the general phenomenon}\label{conclusion} It is to be expected that when $\phi$ is of the form $P/(1-z)^m$, with some polynomial $P$ of degree at most $m$, then the matrix $(I-B)^{2m}M$ will be expressed as the sum of two matrices $T$ and $N$. There, $T$ would be an upper triangular $(2m+1)-$diagonal matrix whose entries are determined, on each diagonal, by polynomials of degree $2(m-1)$. The perturbation matrix $N$ would contain non-null entries only on its last $2m+1$ rows. This structure should be exploited, in a general situation, to make an efficient algorithm to construct a closed formula for all orthonormal polynomials for a particular $\phi$ in finite time. The terms of the recurrence relations will, unfortunately, cease to be constant as soon as we leave the simplest cases but they will, nevertheless, provide a recursive definition for the coefficients $c_k$ of the polynomials $p_n$ for general $k$.

\bibliographystyle{plain}
\bibliography{mybibliography}

\end{document}